\documentclass[twoside,12pt]{amsart}
\usepackage[T1]{fontenc}
\usepackage{amsfonts}
\usepackage{amssymb}
\usepackage{amsthm}
\usepackage{amsmath}
\usepackage{mathrsfs} 
\usepackage{color}
\usepackage{graphicx}
\usepackage{caption}
\usepackage{url}
\usepackage[all]{xy}
\usepackage{lscape}
\theoremstyle{plain}
\usepackage{array}						

\newcommand{\GU}{\mbox{\rm GU}}
\newcommand{\Sp}{\mbox{\rm Sp}}
\newcommand{\SO}{\mbox{\rm SO}}
\newcommand{\GL}{\mbox{\rm GL}}

\newcommand{\End}{\mbox{\rm End}}
\newcommand{\Hom}{\mbox{\rm Hom}}
\newcommand{\Infl}{\mbox{\rm Infl}}
\newcommand{\Ind}{\mbox{\rm Ind}}
\newcommand{\Res}{\mbox{\rm Res}}
\newcommand{\Fix}{\mbox{\rm Fix}}

\newcommand{\cB}{\mathcal{B}}

\newcommand{\cF}{\mathcal{F}}
\newcommand{\cG}{\mathcal{G}}
\newcommand{\cH}{\mathcal{H}}

\newcommand{\cK}{\mathcal{K}}

\newcommand{\cO}{\mathcal{O}}

\newcommand{\cU}{\mathcal{U}}

\newcommand{\Z}{\mathbb{Z}}

\newcommand{\N}{\mathbb{N}}
\newcommand{\Q}{\mathbb{Q}}

\newcommand{\ba}{\mathbf{a}}

\newcommand{\bc}{\mathbf{c}}

\newcommand{\bm}{\mathbf{m}}
\newcommand{\bn}{\mathbf{n}}

\newcommand{\bs}{\mathbf{s}}

\newcommand{\la}{\lambda}

\newcommand{\eps}{\varepsilon}

\newcommand{\ka}{\kappa}

\newcommand{\La}{\Lambda}

\newcommand{\De}{\Delta}

\newcommand{\bla}{\boldsymbol{\la}}
\newcommand{\bnu}{\boldsymbol{\nu}}
\newcommand{\bmu}{\boldsymbol{\mu}}

\newcommand{\te}{\tilde{e}}
\newcommand{\tf}{\tilde{f}}

\newcommand{\lra}{\longrightarrow}

\newcommand{\Ue}{\mathcal{U}_v (\widehat{\mathfrak{sl}_e})}

\newcommand{\bemptyset}{\boldsymbol{\emptyset}}
\newcommand{\charac}{\text{\rm char}}

\newcommand{\mand}{\quad\text{and}\quad}
\newcommand{\Irr}{\text{\rm Irr}}

\newcommand{\fB}{\mathfrak{B}}

\newenvironment{prf}{{\bf Proof.}}{\hfill $\Box$ \\[-1.0ex]}

\newtheorem{num}{Notation}[section]
\newtheorem{define}[num]{Definition}
\newtheorem*{define*}{Definition}

\newtheorem{thm}[num]{Theorem}
\newtheorem*{thm*}{Theorem}
\newtheorem{lem}[num]{Lemma}
\newtheorem*{lem*}{Lemma}
\newtheorem{prp}[num]{Proposition}
\newtheorem*{prp*}{Proposition}
\newtheorem{cor}[num]{Corollary}
\newtheorem*{cor*}{Corollary}

\newtheorem*{conj*}{Conjecture}

\newtheorem*{xmpl*}{Example}
\newtheorem{rem}[num]{Remark}
\newtheorem*{rem*}{Remark}

\begin{document}
\date{\today}
\title{Branching graphs for finite unitary groups in non-defining characteristic}
\author{Thomas Gerber and Gerhard Hiss}


\address{Lehrstuhl D f\"ur Mathematik, RWTH Aachen University,
52062 Aachen, Germany}

\email{gerber@math.rwth-aachen.de}
\email{gerhard.hiss@math.rwth-aachen.de}

\subjclass[2000]{20C33, 20C08, 20G42, 17B37}
\keywords{Harish-Chandra series, endomorphism algebra, Iwahori-Hecke algebra, 
branching graph, unitary group, Fock space, crystal graph}

\begin{abstract}
We show that the modular branching rule (in the sense of Harish-Chandra) on unipotent modules
for finite unitary groups is piecewise described by particular connected components of the 
crystal graph of well-chosen Fock spaces,
under favourable conditions.
Besides, we give the combinatorial formula to pass from one to the other
in the case of modules arising from cuspidal modules of defect $0$.
This partly proves a recent conjecture of Jacon and the authors.
\end{abstract}

\maketitle

\section{Introduction}
\markright{BRANCHING GRAPHS FOR UNITARY GROUPS}

In \cite{GerberHissJacon2014},
Nicolas Jacon and the authors have presented several conjectures 
about the distribution of the unipotent modules for finite
unitary groups based on the concept of weak Harish-Chandra series.
The present paper is a sequel to that work, 
complementing it in several ways. In particular we take some steps 
towards proving the main conjecture stated there, 
namely \cite[Conjecture 5.7]{GerberHissJacon2014}.

Our first objective is the description of the branching graph of a series
of Ariki-Koike algebras $\mathcal{H}_{k,d,n}$ over fields~$k$ of positive 
characteristic~$\ell$, where~$n$ varies and where $k$,~$d$ and the 
parameters are fixed. Ariki has shown \cite[Theorem 6.1]{Ariki2007} that this 
branching graph is equal to the crystal graph of a Fock space representation 
of a quantum algebra of affine type~$A$. Ariki's result uses a version of the
Fock space defined in \cite{JMMO1991} leading to a labelling
of the vertices of the crystal graph by Kleshchev multipartitions.
To apply this result to the conjectures of \cite{GerberHissJacon2014},
we need Uglov's realization of the crystal graph. In Section~\ref{BranchingCyclotomicHecke} of
our paper we review Ariki's result and discuss the relation between the
crystal graphs arising from either Kleshchev's or Uglov's 
realization. We also comment on the connection to canonical basic sets,
thus giving a further motivation for the preference of Uglov's version.

In Section~\ref{BranchingEndomorphism} we first recall the definition of weakly cuspidal pairs from
\cite{GerberHissJacon2014} and the corresponding endomorphism algebras.
The main statement here is Proposition~\ref{HCRestricitonAndRestriction}, which 
extends a result by Howlett and Lehrer; it essentially shows
that the covariant Hom-functors commute with Harish-Chandra restriction
and the restriction in the endomorphism algebras, respectively. 

These results are applied in Section~\ref{HCBranchingGraph} to the Harish-Chandra branching graphs 
of the unitary and symplectic groups and the orthogonal groups of odd degree. 
Such a graph is defined with respect to a series of groups $G_0 \hookrightarrow 
G_1 \hookrightarrow \cdots \hookrightarrow G_n \hookrightarrow \cdots$, 
where~$G_n$ is one of the classical groups above and~$G_{n-1}$ is the Levi 
subgroup of the stabilizer in~$G_n$ of an isotropic vector. A further ingredient
is an algebraically closed field~$k$ of characteristic~$\ell$ different from the
defining characteristic of the groups~$G_n$.
The connected components of a Harish-Chandra branching graph correspond
to the weak Harish-Chandra series of the groups $G_{m + n}$, $n \geq 0$, arising 
from a fixed weakly cuspidal pair $(G_m,X)$ (Proposition~\ref{ConnectedComponents}). 
Define $\mathscr{H}_n$ as the endomorphism algebra of the module obtained by 
Harish-Chandra inducing~$X$ from $G_m$ to~$G_{m + n}$. In Proposition~\ref{IsomorphicBranching} 
we prove that the corresponding family of Hom-functors yields an isomorphism 
between the connected component of the Harish-Chandra branching graph arising from
$(G_m,X)$ and the branching graph of the family $\mathscr{H}_n$, $n \geq 0$ of 
endomorphism algebras.

We expose in Section~\ref{unitarygroups} consequences of these results 
for the Harish-Chandra branching graphs
of the unitary groups. Provided the algebras $\mathscr{H}_n$ 
are Iwahori-Hecke algebras of type~$B_n$ 
with a particular pair of parameters, the Harish-Chandra branching graphs are 
isomorphic to crystal graphs as in \cite[Conjecture~$5.7$]{GerberHissJacon2014};
this is Proposition~\ref{HCGraphIsCrystalGraph}. This result does not yet, however, 
yield the conjectured matching of the vertices of the two graphs involved.
Under the same hypothesis we obtain a strong condition on the structures of 
Harish-Chandra restricted unipotent modules in a minimal situation.
Each direct summand of such a module has a simple socle and a simple head 
isomorphic to each other (Proposition~\ref{SimpleSocle}). Ultimately, this
derives from a result of Grojnowski and Vazirani (see 
\cite[Theorem B]{GrojnowskiVazirani2001}).
Finally, we prove \cite[Conjecture~$5.7$]{GerberHissJacon2014} for the 
principal series and the other series arising from cuspidal unipotent defect~$0$ 
modules (Theorem~\ref{MainTheoremUnitaryGroups}). It is remarkable that these 
results do not require any restriction on~$\ell$. 

\section{The branching graph of Ariki-Koike algebras}
\label{BranchingCyclotomicHecke}

In this section we are interested in the combinatorial description of the 
branching graph for modular Ariki-Koike algebras over a field of positive 
characteristic~$\ell$ in terms of crystals of Fock spaces. This is achieved via Ariki's 
classic categorification theorem, more precisely with the results of his 
paper~\cite{Ariki2007}. The only slight adjustement here is that instead of 
using Kleshchev's realization of the Fock space crystal, as is done by Ariki,
we favor Uglov's version. This is motivated by the fact that we expect 
\cite[Conjecture 5.7]{GerberHissJacon2014} to hold for Uglov's realization.
Throughout this section modules are left modules, and we write $A$-mod for
the category of finitely generated left modules of the algebra~$A$.

\subsection{Ariki-Koike algebras}
\label{defAKalgebra}

Let $d\in\Z_{>0}$, $n\in\Z_{\geq0}$, and let~$k$ be a field. Let $u,v_1, \dots ,
v_d \in k$ with $u$ non-zero. Following \cite{Mathas2004}, we define the 
Ariki-Koike algebra with parameters $u,v_1,\dots,v_d$ to be the $k$-algebra 
$\cH_{k,d,n}$ defined by generators $T_0, T_1, \dots, T_{n-1}$ and relations:
\begin{itemize}
\item the braid relations of type $B$,
\item the relations $(T_0-v_1) \dots (T_0-v_d)=0$ and 
$(T_i-u)(T_i+1)=0$ for all $i = 1, \dots , n-1$.
\end{itemize}
In particular, if $d=2$, then $\cH_{k,d,n}$ is an Iwahori-Hecke algebra of 
type~$B_n$ with parameters $u$ and $-v_1v_2^{-1}$ in the sense of 
\cite[Definition~$4.4.1$ and Remark~$8.1.3$]{GeckPfeiffer2000}, via the change 
of generators $T_i'= T_i$ for $i=1,\dots,n-1$ and $T_0'=- v_2^{-1} T_0$.
This is of importance since we will use this identification from 
Proposition~\ref{HCGraphIsCrystalGraph} on.

If $\cH_{k,d,n}$ is semisimple, there is a labelling of the simple 
$\cH_{k,d,n}$-modules by $d$-partitions of~$n$. The problem of labelling 
the simple modules in the non-semisimple case is much more complicated.
We recall some facts in Sections \ref{branchingAK} and \ref{uglov} below.
We also know in particular that $\cH_{k,d,n}$ is non-semisimple as soon as 
there exist integers $s_1,\dots, s_d$ such that $v_i=u^{s_i}$ for all 
$i=1,\dots,d$, see e.g. \cite[Corollary 3.3]{Mathas2004}. In case
$\cH_{k,d,n}$ is non-semisimple, there is a decomposition map and a 
decomposition matrix relating its representation theory to that of the 
generic (hence semisimple) Ariki-Koike algebra; see 
\cite[Sections 13.3 and 13.4]{Ariki2002} for details.

\subsection{Crystal of the Fock space}
\label{fockspace}

For~$v$ an indeterminate, $e\in\Z_{>1}$ and $\bs=(s_1,\dots,s_d)\in\Z^d$, 
we consider the level~$d$ Fock space representation of~$\Ue$ with 
charge~$\bs$, denoted by $\cF_{\bs,e}$, see for instance 
\cite[Chapter~6]{GeckJacon2011}. It is the $\Q(v)$-vector space with basis 
all $d$-partitions. We know in particular that $\cF_{\bs,e}$ is an integrable 
representation, and is therefore endowed with a crystal structure.

Note that the action of~$\Ue$ on $\cF_{\bs,e}$ requires an order on the 
nodes of a $d$-partition. There are essentially two different orders that 
yield two isomorphic $\Ue$-module structures on $\cF_{\bs,e}$, which we 
recall here. The first one is defined as follows. 

Let $(a,b,c)$ and $(a',b',c')$ be nodes of (the Young diagram of) a 
$d$-partition $\bla$,
\footnote{Here, $c\in\{1,\dots,d\}$ stands for the component of the node, 
$a$ for the row of the node in its component, and $b$ for the column of the node 
in its component.}
such that $b-a+s_c = b'-a'+s_{c'} \mod e$.
We write
$$(a,b,c)\prec_\cU (a',b',c') \text{\quad if \quad} \left\{
\begin{array}{l}b-a+s_c < b'-a'+s_{c'} \text{ \quad or }\\ 
b-a+s_c = b'-a'+s_{c'} \mand c>c'.
\end{array}
\right.$$
We refer to the module structure afforded by this order as 
\textit{Uglov's realization} of the Fock space. This is the order originally 
used in \cite{JMMO1991}, and then in \cite{FLOTW1999}, \cite{Uglov1999}, 
\cite{GeckJacon2011}.

The second order is defined by
$$(a,b,c)\prec_\cK (a',b',c') \text{\quad if \quad} \left\{
\begin{array}{l}c>c' \text{ \quad or }\\ 
c = c' \mand a > a',
\end{array}
\right.$$
and we call \textit{Kleshchev's realization} the module structure afforded by 
this order. This is used in particular in \cite{Ariki2002}, \cite{Mathas2004}, 
\cite{BrundanKleshchev2009}.

\begin{rem}\label{klorder}
{\rm
Note that the definition of the second order does not require the charge~$\bs$.
However, we always want to compare nodes such that $b-a+s_c = 
b'-a'+s_{c'} \mod e$. This means that the order $\prec_\cK$ is in particular 
invariant under translation of any component of~$\bs$ by a multiple of~$e$.
On the other hand, the order~$\prec_\cU$ strictly depends on~$\bs$.
}
\end{rem}

The two orders yield isomorphic Kashiwara crystals (in the sense that the 
two crystal graphs are the same as colored oriented graphs up to a relabelling 
of the vertices), where the action of the crystal operators corresponds to 
adding, respectively removing a so-called good node. Denote by~$\cG_{\bs,e}$ 
the crystal graph of the Fock space, and~$\cB_{\bs,e}$ the connected component 
of this graph containing the empty $d$-partition~$\bemptyset$.

Then~$\cB_{\bs,e}$ is the crystal graph of the irreducible highest weight 
sub-representation $V_{\bs,e}=\Ue.\bemptyset$ of $\cF_{\bs,e}$. The vertices 
appearing in Uglov's (respectively Kleshchev's) realization of the crystal 
graph~$\cB_{\bs,e}$ are called the \textit{Uglov} (respectively 
\textit{Kleshchev}) \textit{$d$-partitions}, and denoted by~$\cU_{\bs,e}$ 
(respectively~$\cK_{\bs,e}$). We call \textit{rank} of a $d$-partition the 
total number of nodes it contains. For~$n$ a fixed integer, we denote by 
$\cU_{\bs,e}(n)$ (respectively~$\cK_{\bs,e}(n)$) the Uglov (respectively 
Kleshchev) $d$-partitions of rank~$n$.

\begin{rem}\label{kl=ug}
{\rm
According to \cite[Proposition 5.1]{Gerber2014}, the orders~$\prec_\cU$ 
and~$\prec_\cK$ coincide on the nodes of a $d$-partition of rank~$n$, if 
and only if $s_i-s_j \geq n-e+1$ for all $1\leq i<j \leq d$.
In particular, if we denote $M=\min\{s_i-s_j \mid i<j\}$, this implies that 
the two crystal graphs are the same up to rank $M+e-1$ (provided of course
that $M+e-1\geq0$).
}
\end{rem}

Algebraic interpretations aside (which are exposed in Section \ref{uglov}),
one can first notice that there is an advantage to work with Uglov rather than 
Kleshchev $d$-partitions. In the particular case where $0\leq s_1 \leq \dots 
\leq s_d \leq e-1$, the Uglov $d$-partitions are known as FLOTW $d$-partitions 
and have a non-recursive combinatorial characterisation, by 
\cite[Theorem 2.10]{FLOTW1999}. In general, there also exists a pathfinding-free 
combinatorial characterisation of Uglov $d$-partitions, see 
\cite[Theorem 6.3]{Gerber2014a}. However, finding a non-recursive 
characterisation of the Kleshchev $d$-partitions in the general case is still 
an open problem (though when $d=2$, this can be achieved via the results of 
\cite[Section 9]{ArikiKreimanTsuchioka2007}).

\subsection{The branching rule for modular Ariki-Koike algebras}
\label{branchingAK}

Let $\cH_n=\cH_{k,d,n}$ be a non-semisimple Ariki-Koike algebra such that~$u$ 
has order~$e$ in~$k$, and $v_i=u^{s_i}$ for all $i=1,\dots,d$ for some 
$(s_1,\dots,s_d)=\bs\in\Z^d$. Following Ariki's book 
\cite[Section 13.6]{Ariki2002}, there exist $i$-restriction and $i$-induction 
functors (for $i=0,\dots,e-1$) which refine the restriction 
(respectively induction) functors between $\cH_{n+1}$-mod and $\cH_n$-mod
(respectively between $\cH_{n}$-mod and $\cH_{n+1}$-mod),
denoted by $i$-$\Res_n^{n+1}$ and $i$-$\Ind_n^{n+1}$.
Then Grojnowski and Vazirani \cite[Theorem B]{GrojnowskiVazirani2001} 
have proved that the functors~$\te_i$ and~$\tf_i$ defined by
$$\begin{array}{ll} 
\te_i M = \text{Soc} (i\text{-}\Res_n^{n+1} (M)) & \text{\quad for } M\in \cH_{n+1}\mbox{\rm -mod}, \\
\tf_i M = \text{Hd} (i\text{-}\Ind_n^{n+1} (M)) & \text{\quad for } M\in \cH_{n}\mbox{\rm -mod}
\end{array}$$
send simple modules to simple modules. This yields a coloring of the arrows of 
the branching graph of $\cH_n$, $n \geq 0$. In fact, as was shown by Ariki (see 
\cite[Theorem 4.1]{Ariki2007}), it even defines the structure of an abstract crystal 
in the sense of \cite[Section 7.2]{Kashiwara1995}
on the set $\Irr(\cH)=\bigsqcup_{n\in\Z_{\geq0}}\Irr(\cH_n)$. 

Using the cellular approach, there is a natural parametrisation of $\Irr(\cH_n)$ 
by the set $\cK_{\bs,e}(n)$ of Kleshchev $d$-partitions of rank $n$; this is the 
main result of \cite{Ariki2001}. Hence, write $\Irr(\cH_n)=\{ D^{\bla} \mid  
\bla\in\cK_{\bs,e}(n) \}$. Actually, we have more. The following result is 
due to Ariki. Importantly, it holds regardless of the characteristic of the 
field~$k$. 
\footnote{This is important to notice, since Ariki's theorem in its general 
version only holds for Ariki-Koike algebras defined over a field of 
characteristic zero.}

\begin{thm}\label{thmbranchingAKkl}
{\rm \cite[Theorem 6.1]{Ariki2007}}
Under the identification $D^{\bla}\leftrightarrow\bla$, the branching graph 
on $\Irr(\cH)$ is exactly the crystal graph $\cB_{\bs,e}$ in Kleshchev's 
realization.
\end{thm}

Because of the discussion of Section~\ref{fockspace}, there is a crystal 
isomorphism~$\varphi$ between the crystal graphs~$\cB_{\bs,e}$ in Kleshchev's 
and Uglov's realization:
$$
\begin{array}{cccc}
\varphi : & \cK_{\bs,e} & \overset{\sim}{\lra} & \cU_{\bs,e} \\
& \bla & \longmapsto & \varphi(\bla).
\end{array}
$$
By definition,~$\varphi$ preserves the rank, and ~$\varphi$ is the identity up 
to rank $M+e-1$ by Remark~\ref{kl=ug}. However, it appears to be difficult to 
determine~$\varphi$ in general.

The following statement is then straightforward from 
Theorem~\ref{thmbranchingAKkl}.

\begin{cor}\label{corbranchingAKug}
Under the identification $D^{\bla}\leftrightarrow\varphi(\bla)$, the 
branching graph on~$\Irr(\cH)$ is exactly the crystal graph~$\cB_{\bs,e}$ in 
Uglov's realization.
\end{cor}

In particular, we get a labelling 
$$\begin{array}{ccc}
\cU_{\bs,e}(n) & \longleftrightarrow & \Irr(\cH_n) \\
\bmu & \longleftrightarrow & C^{\bmu}
\end{array}
$$
with $C^{\bmu}=D^{\bla}$ if and only if $\bmu=\varphi(\bla)$.

\subsection{Compatibility with the theory of canonical basic sets in 
characteristic zero}
\label{uglov}

Up to now, this labelling of~$\Irr(\cH_n)$ by~$\cU_{\bs,e}(n)$ may seem a bit 
superficial. Actually, it is not, since this class of $d$-partitions naturally 
appear in the context of canonical basic sets for~$\cH_n$. For this section, 
we mainly refer to \cite[Chapter 6]{GeckJacon2011}. The theory of canonical 
basic sets provides a way to label the simple modules of a Hecke algebra. 
In fact, this is the suitable labelling for our purpose, see 
Theorem \ref{MainTheoremUnitaryGroups}. 
In the case of~$\cH_n$, this labelling is achieved as follows.  

Recall that to~$\cH_n$ is associated a charge $\bs=(s_1,\dots,s_d)\in\Z^d$, and 
the integer~$e$. 
According to~\cite{Jacon2004}, there is a generalization to~$\mathcal{H}_n$ of Lusztig's 
$\ba$-function for Iwahori-Hecke algebras, 
depending on a parameter~$\bm\in\Q^d$ and denoted by $\ba^\bm$. This induces an 
order~$<_\bm$ on the ordinary Specht modules (i.e.\ on $d$-partitions of~$n$), 
by setting $\bla<_\bm\bmu$ if and only if $\ba^\bm(\bla)<\ba^\bm(\bmu)$. 
If the decomposition 
matrix of~$\cH_n$ is unitriangular with respect to this order (for the exact
definition see \cite[Definition 5.5.19]{GeckJacon2011}), then this gives a 
labelling of the simple modules of~$\cH_n$ by a subset of $d$-partitions of~$n$,
which we call the \textit{canonical basic set} for~$\cH_n$ with respect 
to~$<_\bm$. The following result is due to Geck and Jacon.

\begin{thm}\label{thmcbs}
{\rm \cite[Theorem 6.7.2]{GeckJacon2011}}
Suppose that $\charac(k)=0$.
Let $\bm=(m_1,\dots,m_d)$ such that $0< (s_j-m_j)-(s_i-m_i)<e$ for $i<j$.
Then~$\cU_{\bs,e}(n)$ is the canonical basic set for~$\cH_n$ with respect 
to~$<_\bm$.
\end{thm}

The proof requires Ariki's theorem to identify the decomposition matrix
with the specialistation at $v=1$ of the matrix of the canonical basis 
of $V_{\bs,e} \leq \cF_{\bs,e}$, whence the restriction to the zero 
characteristic.

Write
$$\begin{array}{ccc}
\cU_{\bs,e} & \longleftrightarrow & \Irr(\cH_n) \\
\bmu & \longleftrightarrow & M^{\bmu}
\end{array}
$$
for the labelling given by this theorem. This way of 
labelling~$\Irr(\cH_n)$ does not a priori give any information about the 
branching. However, it is indeed compatible with the crystal structure of 
Corollary~\ref{corbranchingAKug}.

\begin{prp}\label{propcbs}
Suppose that $\charac(k)=0$. Then
for all $\bmu\in\cU_{\bs,e}$, we have $M^{\bmu}=C^{\bmu}$.
\end{prp}
\begin{proof}
Recall that the labelling by~$C^{\bmu}$ is given through the labelling by 
Kleshchev $d$-partitions (arising from Ariki's use of the cellular theory), 
i.e.\ $C^{\bmu}=D^{\bla}$ with $\bmu=\varphi(\bla)$ and $\bla\in\cK_{\bs,e}$.
Now, for fixed~$n$, choose a charge $\bs'=(s'_1,\dots,s'_d)\in\Z^d$ such that 
$s'_i=s_i+t_ie$ for some $(t_1,\dots,t_d)\in\Z^d$ and $s'_i-s'_j \geq n-e+1$ 
for all $1\leq i<j \leq d$. By combining Remarks \ref{klorder} and \ref{kl=ug}, 
we have $\cK_{\bs,e}(n)=\cK_{\bs',e}(n)=\cU_{\bs',e}(n)$. Moreover, it is clear that if we 
write~$\cH'_n$ for the Ariki-Koike algebra with parameters~$u$ a primitive 
$e$-th root of~$1$ and $v_i=u^{s'_i}$, then $\cH'_n=\cH_n$.

Let $\bm=(m_1,\dots,m_d)$ such that $0< (s_j-m_j)-(s_i-m_i)<e$ for $i<j$ and
$\bm'=(m'_1,\dots,m'_d)$ such that $0< (s'_j-m'_j)-(s'_i-m'_i)<e$ for $i<j$.
By Theorem~\ref{thmcbs}, we know that~$\cU_{\bs,e}(n)$ 
(respectively~$\cK_{\bs,e}(n)$) is the canonical basic set for~$\cH_n$ with 
respect to~$<_\bm$ (respectively~$<_{\bm'}$).

Denote $\left\{ G_v(\bmu,\bs) \mid \bmu\in\cU_{\bs,e} \right\}$ the 
canonical basis of~$V_{\bs,e}$. Similarly, denote $\left\{ G'_v(\bla,\bs') \mid
\bla\in\cK_{\bs,e} \right\}$ the canonical basis of~$V_{\bs',e}$.
Decompose the elements 
$$G_v(\bmu,\bs)=\sum_{\bnu\vdash_d |\bmu|} g_{\bnu,\bmu}(v) \bnu \mand 
G'_v(\bla,\bs')=\sum_{\bnu\vdash_d |\bla|} g'_{\bnu,\bla}(v) \bnu$$
on the basis of $d$-partitions. Ariki's theorem ensures that the decomposition 
numbers of~$\cH_n$ are given by the evaluations~$g_{\bnu,\bmu}(1)$ (or 
equivalently~$g'_{\bnu,\bla}(1)$).

Let us focus on~$\cK_{\bs,e}(n)$. The fact that it is the canonical basic set 
with respect to~$<_{\bm'}$ means that for all $\bla\in\cK_{\bs,e}(n)$, $\bla$ 
is the unique element which is smaller (with respect to~$<_{\bm'}$) than all 
$d$-partitions $\bnu$ with $g_{\bnu,\bla}(v)\neq0$. Moreover, 
$g_{\bla,\bla}(1)=1$. Now, because of the particular value of~$\bm'$ we have 
chosen, the order~$<_{\bm'}$ coincides with the classic dominance order, see e.g.
\cite[Proof of Proposition 1.2.11]{Gerber2014b}.
Besides, the same properties holds in the cellular theory for this dominance 
ordering, namely~$\bla$ is the unique element which is smaller (with respect 
to the dominance order) than all $d$-partitions~$\bnu$ such that the 
decomposition number $d_{\bnu,\bla}$ is non-zero; and moreover 
$d_{\bla,\bla}=1$. For this, see for instance to \cite[Theorem 2.2]{Ariki2001}.
This means that the labelling of~$\Irr(\cH_n)$ by~$\cK_{\bs,e}(n)$ as in
Theorem \ref{thmbranchingAKkl} on the one hand and by the theory of canonical 
basic sets on the other hand is exactly the same. Precisely, if 
$\bla\in\cK_{\bs,e}(n)$ labels the simple module~$M$ by the theory of canonical 
basic sets, then $M=D^{\bla}$. In particular, the theory of canonical basic sets 
provides the same branching information.

Finally, let us go back to the labelling of~$\Irr(\cH_n)$ by~$\cU_{\bs,e}(n)$ 
via the theory of canonical basic sets. We already know the existence of the 
crystal isomorphism $\cK_{\bs,e} \overset{\varphi}{\lra} \cU_{\bs,e}$. The 
characteristic property of the canonical basis elements~$G_v(\bmu,\bs)$ 
and~$G'_v(\bla,\bs')$ \cite[Section 4]{Uglov1999} ensures that this crystal 
isomorphism maps the Kleshchev $d$-partition~$\bla$ labelling an irreducible 
module~$M$ to the Uglov $d$-partition~$\varphi(\bla)$ labelling the same 
module~$M$ (which we had denoted~$M^{\varphi(\bla)}$). Since $M=D^{\bla}$, we 
deduce $M^{\bmu} = M =  D^{\bla} = C^{\bmu}$ where $\bmu = \varphi(\bla)$.
\end{proof}

We end this section by mentioning that it is also relevant to work with Uglov's 
realization of the crystal when we want to link the representation theory of 
Ariki-Koike algebras with that of rational Cherednik algebras. 
In~\cite{Shan2008}, Shan has defined $i$-induction and $i$-restriction functors 
on the category~$\cO_c$ for the family of rational Cherednik algebras of 
type~$G(d,1,n)$ with parameter~$c$, for $n\geq0$. Similarly to Ariki's $i$-induction and 
$i$-restriction (for the Ariki-Koike algebra), she has proved that these 
operators induce the structure of an abstract crystal, which is isomorphic to 
that of a Fock space~$\cF_{\bs,e}$, where the parameters~$\bs$ and~$e$ are 
determined by~$c$. Moreover, Losev~\cite{Losev2013} has used Uglov's realization 
of~$\cF_{\bs,e}$ to give an explicit combinatorial rule for the computation of 
this crystal.

Now, the representation theory of Ariki-Koike algebras on the one hand, and 
rational Cherednik algebras on the other hand, are known to be related by an 
exact functor $KZ : \cO_c \lra \bigoplus_{n \geq 0} \cH_n$-mod. This functor 
has the nice property of 
commuting with Ariki's (respectively Shan's) $i$-induction and $i$-restriction 
functors. Therefore, it maps the Uglov crystal of~$\cF_{\bs,e}$ encoding Shan's 
branching rule to the Uglov crystal~$\cB_{\bs,e}$ encoding Ariki's branching 
rule. Note that relying on previous knowledge 
(\cite[Corollary 5.8]{ChlouverakiGordonGriffeth2011}) about canonical basic 
sets, this was already mentioned in \cite[Paragraph 4.15]{GordonMartino2013}.

\section{Branching in endomorphism algebras}
\label{BranchingEndomorphism}

Let~$G$ be a finite group with a split $BN$-pair of characteristic~$p$
satisfying the commutator relations. The set of $N$-conjugates of the 
standard Levi subgroups is denoted by~$\mathcal{L}$. We let $\mathcal{L'} :=
\mathcal{L'}_G$ denote an $N$-stable subset of~$\mathcal{L}$ containing~$G$
and $B \cap N$ and satisfying 
$L \cap {^x\!M} \in \mathcal{L}'$ for all $L, M \in \mathcal{L'}$ and all 
$x \in N$. For $M \in \mathcal{L'}$ we put $\mathcal{L'}_M := \{ L \in 
\mathcal{L'} \mid L \leq M \}$.

Next, let~$k$ be a field of characteristic different from~$p$, such that~$k$
is a splitting field for all subgroups of~$G$. If~$A$ is a $k$-algebra, we write 
$A$-mod and mod-$A$ for the category of finite dimensional left, respectively 
right, $A$-modules. 

For $L \in \mathcal{L}$ we denote $R_L^G$ and ${^*\!R}_L^G$ the Harish-Chandra
induction and restriction functors, respectively. These are defined using a
choice of a parabolic subgroup of~$G$ with Levi complement~$L$, but are
known to be independent of this choice up to a natural isomorphism of functors
(see e.g.\ \cite[Theorem~$3.10$]{CabanesEnguehard2004}).

Following \cite[Section 2.3]{GerberHissJacon2014}, we call a simple 
$kG$-module~$X$ weakly cuspidal (with 
respect to~$\mathcal{L}'$), if ${^*\!R}_L^G(X) = 0$ for all $L \in \mathcal{L}'$ 
with $L \neq G$. A weakly cuspidal pair $(L,X)$ consists of an $L \in 
\mathcal{L'}$ and a simple $kL$-module~$X$ which is weakly cuspidal with respect 
to $\mathcal{L'}_L$. 

Now fix a weakly cuspidal pair $(L,X)$ and put 
$$Y := R_L^G(X).$$
Write $H_G := \End_{kG}( Y )$ and $kG\mbox{\rm -mod}_Y$
for the full subcategory of $kG$-mod consisting of those modules which are
quotients and submodules of a finite direct sum of copies of~$Y$. The 
covariant Hom-functor with respect to~$Y$ is denoted by~$F_Y$:
$$F_Y: kG\mbox{\rm -mod} \rightarrow \mbox{\rm mod-}H_G,\quad V \mapsto 
\Hom_{kG}(Y,V).$$
Here, $H_G$ acts on the right of $\Hom_{kG}(Y,V)$ by composition of maps.
A result of Cabanes states that $F_Y$ induces an equivalence 
between $kG\mbox{\rm -mod}_Y$ and $\mbox{\rm mod-}H_G$, provided 
that~$H_G$ is self-injective (see \cite[Theorem 2]{Cabanes1990}).

Suppose now that $M \in \mathcal{L}'$ with $L \leq M$ (then $L \in 
\mathcal{L}'_M$ by definition of $\mathcal{L}'_M$) and put
$$Z := R_L^M(X).$$
Write $H_M := \End_{kM}(Z)$ and~$F_Z$ for the covariant
Hom-functor between $kM\mbox{\rm -mod}$ and $\mbox{\rm mod-}H_M$.
There is a natural embedding 
$$i : H_M \rightarrow \End_{kG}( R_M^G( Z ) ),$$
and as $R_M^G(Z) \cong Y$, the map~$i$
induces a restriction functor
$$\Res^{H_G}_{H_M}: \mbox{\rm mod-}H_G \rightarrow \mbox{\rm mod-}H_M.$$

The following result is essentially due to Howlett and Lehrer (see
\cite[Theorem 1.13]{HowlettLehrer1983}) although they formulate it for the
contravariant Hom-functor and in the semisimple case. 
The analogous statement in the case of defining characteristic is
contained in \cite[2.3]{Cabanes1990}.

\begin{prp}
\label{HCRestricitonAndRestriction}
The following diagram of functors is commutative up to a natural
isomorphism of functors.

$$
\begin{xy}
\xymatrix{
kG\mbox{\rm -mod} \ar[r]^{F_Y} \ar[d]_{{^*R}_M^G} & \mbox{\rm mod-}H_G \ar[d]^{\Res^{H_G}_{H_M}} \\
kM\mbox{\rm -mod} \ar[r]_{F_Z} & \mbox{\rm mod-}H_M
}
\end{xy}
$$
\end{prp}
\begin{prf}
We may assume~$Y = R_M^G(Z)$. To proceed, we choose a parabolic subgroup~$Q$
of~$G$ with Levi complement~$M$ to define the functors $R_M^G$ and ${^*R}_M^G$; 
in particular,
$$Y = R_M^G(Z) = kG \otimes_{kQ} \Infl_M^Q( Z ).$$
Then for every $V \in kG\mbox{\rm -mod}$ we have
$$F_Y(V) = \Hom_{kG}( Y, V ) = \Hom_{kG}( kG \otimes_{kQ} \Infl_M^Q( Z ), V ).$$
Applying Frobenius reciprocity, we find
$$\Hom_{kG}( kG \otimes_{kQ} \Infl_M^Q( Z ), V ) \cong
\Hom_{kQ}( \Infl_M^Q( Z ), \Res^G_Q( V ) )$$
as right $H_M = \End_{kQ}( \Infl_M^Q( Z ) )$-modules. (Indeed, the above
isomorphism equals $\eta^*$, with $\eta : \Infl_M^Q( Z ) 
\rightarrow kG \otimes_{kQ} \Infl_M^Q( Z )$, $z \mapsto 1 \otimes z$.)
Clearly, 
\begin{eqnarray*} \Hom_{kQ}( \Infl_M^Q( Z ), \Res^G_Q( V ) ) & \cong & 
\Hom_{kM}( Z , \Fix_{O_p(Q)}( \Res^G_Q( V ) ) ) \\
& = & F_Z( {^*R}_M^G( V ) )
\end{eqnarray*}
as right $H_M$-modules. As all of the above isomorphisms are natural,
the result follows.
\end{prf}

\begin{lem}
\label{modYstability}
Let $S \in kM\mbox{\rm -mod}_Z$ and $T \in kG\mbox{\rm -mod}_Y$.
Then $R_M^G(S) \in kG\mbox{\rm -mod}_Y$ and ${^*\!R}_M^G(T) \in 
kM\mbox{\rm -mod}_Z$.
\end{lem}
\begin{prf}
As $R_M^G$ is exact, and as $R_M^G(Z)$ is isomorphic to~$Y$, the first assertion 
is clear. To prove the second assertion, assume that~$T$ is a submodule and a 
quotient of $Y' = R_L^G(X')$, where~$Y'$ and~$X'$ are direct sums of the same 
number of copies of~$Y$, respectively~$X$. Hence ${^*\!R}_M^G(T)$ is a submodule 
and a quotient of ${^*\!R}_M^G( R_L^G(X') )$. By Mackey's theorem, the latter is 
isomorphic to a direct sum of modules of the form
\begin{equation}
\label{Summand}
R_{M \cap {^x\!L}}^M ( {^*\!R}_{M \cap {^x\!L}}^{^x\!L}( {^x\!X'} ) )
\end{equation}
for suitable $x\in N$. As $(L,X)$ is weakly cuspidal, so is $({^x\!L},{^x\!X})$
for all such~$x$. It follows that a summand~(\ref{Summand}) is zero unless 
${^x\!L} \leq M$, in which case~(\ref{Summand}) is isomorphic to a direct sum 
of copies of~$Z$. The claim follows.
\end{prf}

\begin{prp}
\label{SameBranching}
Let $S \in kM\mbox{\rm -mod}_Z$ and $T \in kG\mbox{\rm -mod}_Y$. Suppose
that $H_M$ is self-injective. Then
$$\Hom_{H_M}( F_Z(S), \Res^{H_G}_{H_M}( F_Y(T) ) ) \cong \Hom_{kM}( S, {^*R}_M^G(T) ).$$
\end{prp}
\begin{prf}
We have
\begin{eqnarray*}
\Hom_{H_M}( F_Z(S), \Res^{H_G}_{H_M}( F_Y(T) ) ) & \cong &
\Hom_{H_M}( F_Z(S), F_Z( {^*R}_M^G(T) ) ) \\
 & \cong & \Hom_{kM}( S, {^*R}_M^G(T) ),
\end{eqnarray*}
where the first isomorphism follows from 
Proposition~\ref{HCRestricitonAndRestriction}, and the second from 
\cite[Theorem 2]{Cabanes1990} together with Lemma~\ref{modYstability}.
\end{prf}

\begin{rem}
\label{SymmetricRemark}
{\rm
By \cite[Proposition~$2.3$]{GerberHissJacon2014}, the simple submodules of~$Z$
and of~$Y$ belong to $kM\mbox{\rm -mod}_Z$ and $kG\mbox{\rm -mod}_Y$, 
respectively. Thus if~$H_M$ is self-injective, Proposition~\ref{SameBranching} 
applies to these simple modules.
}
\end{rem}

\section{The Harish-Chandra branching graph}
\label{HCBranchingGraph}

In \cite[Section~4]{GerberHissJacon2014}, we introduced the Harish-Chandra 
branching graph for unipotent modules of certain classical groups. Let us 
recall this definition. Fix primes $p \neq \ell$ and a power~$q$ of~$p$. 
Let~$k$ denote an algebraically closed field of characteristic~$\ell$. For 
every $n \in \mathbb{Z}_{\geq 0}$ we consider the groups $\GU_{2n}(q)$, 
$\GU_{2n+1}(q)$, $\SO_{2n+1}(q)$ and~$\Sp_{2n}(q)$ (with $\GU_0(q)$ and 
$\Sp_0(q)$ the trivial group), and call~$n$ the \textit{rank} of such a 
group. The \textit{Dynkin type} of these groups
is ${^2\!}A_0$, ${^2\!}A_1$,~$B$ and~$C$, respectively. 

We now fix one of these Dynkin types~$\mathcal{D}$, say, and write $G_n$
for the group of Dynkin type~$\mathcal{D}$ and rank~$n$. Then~$G_n$ is
a group with a split $BN$-pair of characteristic~$p$ satisfying the 
commutator relations. If~$r$,~$m$ are non-negative integers with $r + m = n$, 
there is a standard Levi subgroup~$L_{r,m}$ of~$G_n$
isomorphic to $G_r \times \GL_1(q^\delta) \times \cdots \times 
\GL_1(q^\delta)$ with~$m$ factors $\GL_1(q^\delta)$ and $\delta = 1$ 
if~$\mathcal{D}$ is~$B$ or~$C$, and $\delta = 2$, otherwise. These Levi
subgroups and their $N$-conjugates are called \textit{pure Levi subgroups} 
of~$G_n$. The set $\mathcal{L}'_n$ of pure Levi subgroups of~$G_n$ is 
$N$-stable and satisfies $L \cap {^x\!M} \in \mathcal{L}'_n$ for all
$L, M \in \mathcal{L}'_n$.

The Harish-Chandra branching graph $\mathcal{G}_{\mathcal{D},q,\ell}$ for the
Dynkin type~$\mathcal{D}$ (and fixed~$q$ and~$\ell$) is defined as follows. 
Its vertices are the 
isomorphism classes of the unipotent $kG_n$-modules, where~$n$ runs over the 
integers. There is an arrow $[X] \rightarrow [Y]$ between the vertices~$[X]$ 
and~$[Y]$ of $\mathcal{G}_{\mathcal{D},q,\ell}$, if and only if there is 
$n \in \mathbb{Z}_{\geq 0}$ such that~$X$ and~$Y$ are $kG_n$- and $kG_{n+1}$-modules, 
respectively, and the inflation of~$X$ to~$L_{n,1}$ occurs in the socle of 
${^*\!R}_{L_{n,1}}^{G_{n+1}}(Y)$, i.e.\ if and only if
$\Hom_{kL_{n,1}}( \Infl_{G_{n}}^{L_{n,1}}( X ), {^*\!R}_{L_{n,1}}^{G_{n+1}}(Y) )
\neq 0$. (Recall that $L_{n,1} \cong G_n \times \GL_1(q^\delta)$.)

Let $(L,X)$ be a weakly cuspidal pair, where $L = G_m$ for some $m \in 
\mathbb{Z}_{\geq 0}$, and~$X$ is unipotent. For a non-negative integer~$n$ put
$$Y_n := R_{L_{m,n}}^{G_{m+n}}( X ),$$
where~$X$ is viewed as a $k{L_{m,n}}$-module via inflation. Moreover, we
put 
$$\mathscr{H}_n := \mathscr{H}_n(X) := \End_{kG_{m+n}}( Y_n ).$$
Then there are natural inclusions
$$\nu_n: \mathscr{H}_n \rightarrow \mathscr{H}_{n+1}.$$
We also write $F_n$ for the $\Hom$-functor 
$$F_n : kG_n\mbox{\rm -mod}_{Y_n} \rightarrow \mbox{\rm mod-}\mathscr{H}_n.$$
The branching graph for $\mathscr{H}_n$, $n \geq 0$, is defined as follows. 
Its vertices are
the isomorphism classes of the simple $\mathscr{H}_n$-modules, where~$n$
runs through the non-negative integers. Two such vertices~$[S]$ and~$[T]$ are
connected by an arrow $[S] \rightarrow [T]$, if and only if there is $n \in
\mathbb{Z}_{\geq 0}$, such that~$S$ is an $\mathscr{H}_n$-module and~$T$ is an
$\mathscr{H}_{n+1}$-module such that~$S$ occurs in the socle of 
$\Res^{\mathscr{H}_{n+1}}_{\mathscr{H}_n}( T )$, i.e.\ if and only if
$\Hom_{\mathscr{H}_n}( S, \Res^{\mathscr{H}_{n+1}}_{\mathscr{H}_n}( T ) ) 
\neq 0$.

\begin{define}
{\rm 
Define $\mathcal{G}_{\mathcal{D},q,\ell}( X )$ to be the induced subgraph of 
$\mathcal{G}_{\mathcal{D},q,\ell}$, consisting of the vertices~$[Y]$ such that 
there is a directed path from~$[X]$ to~$[Y]$. 
}
\end{define}

\begin{prp}
\label{ConnectedComponents}
The subgraph $\mathcal{G}_{\mathcal{D},q,\ell}( X )$ is a connected component
of $\mathcal{G}_{\mathcal{D},q,\ell}$ (with respect to the underlying undirected
graph), and every connected component of $\mathcal{G}_{\mathcal{D},q,\ell}$ is
of this form. 
\end{prp}
\begin{prf}
By definition, every vertex of
$\mathcal{G}_{\mathcal{D},q,\ell}( X )$ is connected to~$[X]$, and 
thus $\mathcal{G}_{\mathcal{D},q,\ell}( X )$ is connected. 

Consider a path $[Y] \rightarrow [Z]$ of length~$1$ in 
$\mathcal{G}_{\mathcal{D},q,\ell}$. 
We claim that~$[Y]$ and~$[Z]$ belong to $\mathcal{G}_{\mathcal{D},q,\ell}( X )$ 
if and only if one of~$[Y]$ or~$[Z]$ belongs to this subgraph. To prove this, it 
suffices to show that~$[Y]$ is a vertex of 
$\mathcal{G}_{\mathcal{D},q,\ell}( X )$ if~$[Z]$ is one. Consider a
source vertex~$[X']$ of $\mathcal{G}_{\mathcal{D},q,\ell}$ such that there is
a directed path from~$[X']$ to~$[Y]$. Suppose that~$X'$ and~$Z$ are unipotent 
modules of~$G_{m'}$ and~$G_{m+n}$, respectively. Then~$Z$ lies in the weak
Harish-Chandra series defined by $(L_{m,n},X)$ and $(L_{m',n'}, X')$ with
$m+n = m'+n'$. It follows that $L_{m,n}$ and $L_{m',n'}$ are conjugate in
$G_{m+n}$ and thus are equal. Moreover, $X$ and $X'$ are conjugate by an
element in the relative Weyl group of $L_{m,n}$. As the latter group fixes~$X$,
it follows that $[X] = [X']$, thus proving our claim.

The claim implies that the connected component of
$\mathcal{G}_{\mathcal{D},q,\ell}$ containing 
$\mathcal{G}_{\mathcal{D},q,\ell}( X )$ is equal to 
$\mathcal{G}_{\mathcal{D},q,\ell}( X )$. 

As every vertex of the Harish-Chandra branching graph belongs to some weak
Harish-Chandra series, it is clear that every connected component of 
$\mathcal{G}_{\mathcal{D},q,\ell}$ is of the asserted form.
\end{prf}

\begin{prp}
\label{IsomorphicBranching}
The collection of functors $F_n$, $n \in \mathbb{Z}_{\geq 0}$ yields an 
isomorphism between $\mathcal{G}_{\mathcal{D},q,\ell}(X)$ and the branching 
graph of $\mathscr{H}_n$, $n \geq 0$.
\end{prp}
\begin{prf}
As already noted in the proof of \cite[Proposition~$2.3$]{GerberHissJacon2014}, 
the general results of Cabanes and Enguehard 
\cite[Theorems 1.20(iv), 2.27]{CabanesEnguehard2004} imply that~$\mathscr{H}_n$ 
is symmetric, hence self-injective, for all non-negative integers~$n$. Our claim 
now follows from Proposition~\ref{SameBranching} and Remark~\ref{SymmetricRemark}.
\end{prf}

By \cite[Theorem 3.2]{GerberHissJacon2014}, we have that $\mathscr{H}_n$ is an
Iwahori-Hecke algebra of type~$B_n$, with parameters~$q^\delta$ and~$Q$, 
where~$Q$ occurs on the doubly laced end node of the Dynkin diagram. Although
the value of~$Q$ can only be determined explicitly in some cases, it is clear
from the proof of \cite[Theorem 3.2]{GerberHissJacon2014} that~$Q$ only depends 
on~$L$ and not on~$n$. 

\section{The unitary groups}
\label{unitarygroups}

As in Section~\ref{HCBranchingGraph}, we fix primes $p \neq \ell$ and a 
power~$q$ of~$p$. Again,~$k$ denotes an algebraically closed field of 
characteristic~$\ell$. We write~$e$ for the order of $-q$ in the finite field 
$\mathbb{F}_{\ell}$, assuming henceforth that~$e$ is odd and larger than~$1$.

Assume that~$\mathcal{D}$ is one the Dynkin types~${^2\!}A_0$ or ${^2\!}A_1$,
Thus, if~$n$ is a non-negative integer,~$G_n$ now denotes one of the groups
$\GU_{2n}(q)$ or $\GU_{2n+1}(q)$. 

We also fix a weakly cuspidal unipotent $kG_m$-module~$X$ for some non-negative
integer~$m$. For every non-negative integer $n$, we view~$X$ as a module for
$L_{m,n}$ via inflation, so that $(L_{m,n},X)$ is a weakly cuspidal pair
in~$G_{m + n}$.

Notice that we have worked with the categories of finitely generated right
$\mathscr{H}_n$-modules for the endomorphism algebras arising in the previous
section. We are now going to apply the results of 
Section~\ref{BranchingCyclotomicHecke} to these algebras, where we have 
worked with left modules. This is no loss since duality of vector spaces
induces an equivalence between $A$-mod and mod-$A$ for a finite-dimensional
$k$-algebra~$A$.

\begin{prp}
\label{HCGraphIsCrystalGraph}
Suppose that for some non-negative integer~$s$ and any non-negative integer~$n$,
$$\mathscr{H}_n := \End_{kG_{m+n}}( R_{L_{m,n}}^{G_{n+m}}( X ) )$$ is an 
Iwahori-Hecke algebra of type~$B_n$ with parameters $q^2$ and $Q = q^{2s+1}$. 

Then the Harish-Chandra branching graph $\mathcal{G}_{\mathcal{D},q,\ell}(X)$
is isomorphic to the crystal graph $\mathcal{B}_{\mathbf{s},e}$, in which the
colors of the arrows are neglected, where 
$\mathbf{s} = (s + (1-e)/2,0)$. (For the notation see 
Section~\ref{BranchingCyclotomicHecke}.2.)
\end{prp}
\begin{prf}
Notice that an Iwahori-Hecke algebra over $k$ of type $B_n$ with parameters
$q^2$ and $q^{2s+1}$ is equal to the Ariki-Koike algebra $\cH_n := 
\mathcal{H}_{k,2,n}$
with parameters $u = q^2$, $v_1 = -q^{2s+1}$ and $v_2 = 1$, as explained                  
in Section~\ref{defAKalgebra}. As~$q$ is a primitive $2e$th root of unity 
in~$k$, we have $v_1 = u^{s + (1-e)/2}$.

By Proposition~\ref{IsomorphicBranching}, the Harish-Chandra branching graph 
is isomorphic to the branching graph of~$\mathcal{H}_n, n\geq0$. 
The latter is isomorphic to the crystal graph $\mathcal{B}_{\mathbf{s},e}$ by 
the results of Section~\ref{BranchingCyclotomicHecke}.3.
\end{prf}

It seems reasonable to expect that the hypothesis of the above proposition is 
always satisfied. Unfortunately, a general proof of this result appears to be 
out of reach at the moment. It would, however, follow from 
\cite[Conjecture~$5.5$]{GerberHissJacon2014} in conjunction with Lemma~$5.5$ 
and Theorem~$3.2$ of \cite{GerberHissJacon2014}, at least if $\ell > 2m + 1$.
Thus, in view of Proposition~\ref{HCGraphIsCrystalGraph}, Conjecture~$5.7$ of 
\cite{GerberHissJacon2014} follows from 
\cite[Conjecture~$5.5$]{GerberHissJacon2014}, up to the labelling of the 
vertices.

A more conceptual approach to proving the latter conjecture, 
again up to the labelling of the vertices,
would be to reveal a categorification phenomenon,
in the very same spirit as \cite{Ariki1996}
for Ariki-Koike algebras, \cite{Shan2008} for cyclotomic rational Cherednik 
algebras, and \cite{BrundanKleshchev2009} for cyclotomic quiver Hecke algebras.
In our case, it is particularly crucial to understand how a coloring
of the arrows in the Harish-Chandra branching graph would arise,
and, in turn, give an interpretation of the rest of the abstract crystal data
(namely the weight and the functions $\varphi_i$ and $\eps_i$ 
in Kashiwara's notation \cite[Section 7]{Kashiwara1995}).
Actually, according to Uglov's work \cite{Uglov1999}, the level $2$ Fock space 
$\cF_{\bs,e}$ can be seen as a submodule of the integrable 
$\cU_v'(\widehat{\mathfrak{sl}_e})$-module $\La^{s+(1-e)/2}$ consisting of
semi-infinite wedge products. Besides, there is an integrable action of 
level~$e$ of $\cU_{v^{-1}}'(\widehat{\mathfrak{sl}_2})$ on $\La^{s+(1-e)/2}$,
as well as an action of a Heisenberg algebra.
These three actions pairwise commute, and each element of $\La^{s+(1-e)/2}$ is 
obtained by acting on some very particular elements (\cite[Theorem 4.8]{Uglov1999}).
It would be interesting to use this approach to the Fock space and to look for 
categorical actions  of these algebras (in the sense of \cite{ChuangRouquier2008}) 
in the context of $kG_n$-modules which have a 
filtration by unipotent modules. 
Note that a categorification 
of the action of the Heisenberg algebra is achieved in \cite{ShanVasserot2012}.

\begin{prp}
\label{SimpleSocle}
Let the notation and assumptions be as in Proposition~\ref{HCGraphIsCrystalGraph}.
Put $G := G_{m+n}$ and $L := L_{m,n}$.
Let~$S$ be a simple $kG$-module in the head (or socle) of $R_L^G( X )$. Let
$M := L_{m+n - 1, 1}$ be the maximal pure Levi subgroup of~$G$. Suppose that
$${^*\!R}_M^G( S ) = S_1 \oplus \cdots \oplus S_r$$
with indecomposable $kM$-modules $S_i$, $1 \leq i \leq r$. Then the socle of 
each~$S_i$ is simple and isomorphic to its head. 
\end{prp}
\begin{prf}
Lemma~\ref{modYstability} implies that each direct summand~$S_i$ is contained 
in $kM\mbox{\rm -mod}_Z$ with $Z = R_L^M( X )$. The assertion now follows from
the corresponding properties of the simple $\mathcal{H}_n$-modules discussed in
Section \ref{branchingAK} (more precisely \cite[Theorem B]{GrojnowskiVazirani2001}),
together with Proposition~\ref{HCRestricitonAndRestriction} and 
\cite[Theorem 2]{Cabanes1990}).
\end{prf}

It follows from Proposition~\ref{HCGraphIsCrystalGraph} that the modules in a
weak Harish-Chandra series can be labelled by Uglov bipartitions, provided the
hypothesis of the proposition is satisfied. The unipotent $kG$-modules of
$G := \GU_r(q)$ are also labelled by partitions of~$r$, as explained in
\cite[Section 5.3]{GerberHissJacon2014}. Let~$\nu$ be a partition of~$r$. We then 
write~$X_\nu$ for the unipotent $kG$-module labelled by~$\nu$. We will now
discuss the question of matching the two labellings in special cases. 

We first give an example illustrating why we want to consider Uglov's crystal 
structure rather than Kleshchev's. In fact, Kleshchev's realization gives rise 
to bipartitions that do not naturally appear as twisted $2$-quotients of the 
labels of the vertices in the Harish-Chandra branching graph.
Recall the combinatorial notions introduced in
\cite[Section 5.2]{GerberHissJacon2014}. For a partition
$\la=(\la_1, \la_2, \dots)$ with $\lambda_1 \geq \lambda_2 \geq \cdots$, let
$\De_t := (t,t-1,\dots, 1) = \la_{(2)}$ be the $2$-core of $\la$, and
$(\la^1,\la^2)=\la^{(2)}$ the $2$-quotient of $\la$. The \textit{twisted
$2$-quotient} of $\la$ is the bipartition $\overline{\la}^{(2)}$ defined by
$$\overline{\la}^{(2)} = \left\{
\begin{array}{ll}
(\la^1,\la^2) & \text{ if $t$ is even} \\
(\la^2,\la^1) & \text{ if $t$ is odd}.
\end{array}
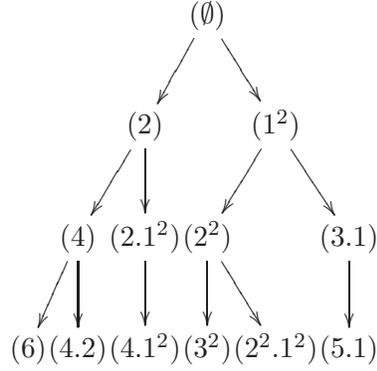
\right.$$

\newcolumntype{M}[1]{>{\centering}p{#1}}

\begin{center}
 
\begin{tabular}{M{5.5cm}|M{5.5cm}}
${\small
\begin{xy}
\xymatrix@C-30pt{
& & & (\emptyset) \ar[dl] \ar[dr] & & \\ 
& & (2) \ar[dl] \ar[d] & &  (1^2) \ar[dl] \ar[dr]  & \\
& (4) \ar[dl] \ar[d] & (2.1^2)  \ar[d] & (2^2) \ar[d] \ar[dr] & & (3.1) \ar[d] \\
(6) & (4.2) & (4.1^2) & (3^2) & (2^2.1^2) & (5.1)
}
\end{xy}
}$

&

$
{\small
\begin{xy}
\xymatrix@C-31pt{
& & & (\emptyset,\emptyset) \ar[dl] \ar[dr] & & \\ 
& & (1,\emptyset) \ar[dl] \ar[d] & &  (\emptyset,1) \ar[dl] \ar[dr]  & \\
& (2,\emptyset) \ar[dl] \ar[d] & (1^2,\emptyset)  \ar[d] & (1,1) \ar[d] \ar[dr] & & (\emptyset,2) \ar[d] \\
(3,\emptyset) & (2,1) & (2.1,\emptyset) & (1,2) & (1,1^2) & (\emptyset,3)
}
\end{xy}
}
$

\tabularnewline

The connected component of the Harish-Chandra branching graph giving the 
principal series for $\GU_r(q)$, $0 \leq r \leq 6$ even, $e=3$, 
under the identification $X_\la\leftrightarrow \la$.

&

The connected component of the Harish-Chandra branching graph giving the 
principal series for $\GU_r(q)$, $0 \leq r \leq 6$ even, $e=3$,
under the identification $X_\la\leftrightarrow \overline{\la}^{(2)}$.

\end{tabular}

\begin{tabular}{M{5.5cm}|M{5.5cm}}
$
{\small
\begin{xy}
\xymatrix@C-31pt{
& & & (\emptyset,\emptyset) \ar[dl] \ar[dr] & & \\ 
& & (1,\emptyset) \ar[dl] \ar[d] & &  (\emptyset,1) \ar[dl] \ar[dr]  & \\
& (2,\emptyset) \ar[dl] \ar[d] & (1^2,\emptyset)  \ar[d] & (1,1) \ar[d] \ar[dr] & & (\emptyset,2) \ar[d] \\
(3,\emptyset) & (2,1) & (2.1,\emptyset) & (1,2) & (1,1^2) & (\emptyset,3)
}
\end{xy}
}
$

&

$
{\small
\begin{xy}
\xymatrix@C-31pt{
& & & (\emptyset,\emptyset) \ar[dl] \ar[dr] & & \\ 
& & (1,\emptyset) \ar[dl] \ar[d] & &  (\emptyset,1) \ar[dl] \ar[dr]  & \\
& (2,\emptyset) \ar[dl] \ar[d] & (1^2,\emptyset)  \ar[d] & (1,1) \ar[d] \ar[dr] & & (\emptyset,2) \ar[d] \\
(3,\emptyset) & (2,1) & (2.1,\emptyset) & \boldsymbol{(1^2,1)} & (1,1^2) & \boldsymbol{(1,2)}
}
\end{xy}
}$

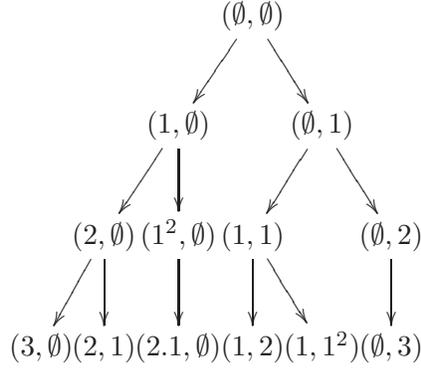
\tabularnewline

The crystal graph $\cB_{\bs,e}$ in Uglov's realization, for $\bs=(-1,0), 
e=3$, up to rank $3$.

&
The crystal graph $\cB_{\bs,e}$ in Kleshchev's realization, for $\bs=(-1,0), 
e=3$, up to rank $3$.

\end{tabular}
\end{center}
The first graph can be derived from the decomposition matrices computed by
Dudas and Malle in \cite{DudasMalle2013}. Notice that the results of Dudas and 
Malle require the condition $\ell > 6$, but $e = 3$ implies $\ell \geq 7$ 
as $2e \mid \ell - 1$. 
The third graph agrees with the second graph (in particular the labels of the 
vertices match). However, Uglov's and Kleshchev's realizations of $\cB_{\bs,e}$ 
already differ in rank~$3$, as illustrated in the fourth graph.
The differences with the Uglov crystal are indicated in boldface type.
Such an example had first been found by Gunter Malle and the second author.

Let us proceed to the main result of this section.
Let $\la=(\la_1, \la_2, \dots)$ with $\lambda_1 \geq \lambda_2 \geq \cdots$
and $2$-core $\De_t := (t,t-1,\dots, 1)$.
Now, let $n(\la) = \sum_{i\geq 1} (i-1)\la_i.$ Further, write 
$\bla=\overline{\la}^{(2)}$, $\bc=(t,0)$ and set 
$\bn(\bla) = \sum_{j\geq1} (j-1) (\ka_j - \ka_j^0),$
where $\ka_j$ (respectively $\ka_j^0$) denote the elements of the symbol 
$\fB(\bla,\bc)$ (respectively $\fB(\bemptyset,\bc)$) in decreasing order, 
see \cite[Section 2.2]{JaconLecouvey2012} for the notation.

\begin{rem}\label{remafunction}
Lusztig's $\ba$-function for the Iwahori-Hecke algebra of type $B_n$
with parameters $q^2$ and $q^{2t+1}$ is the function $\ba:=\ba^\bm$ of 
Section \ref{uglov} with $\bm=(t,0)$, and is given by the formula 
$\ba(\bmu) = 2 \bn(\bmu)$, for all $\bmu\vdash_2 n$. This follows from 
\cite[Proposition 5.5.11, Example 5.5.14 and Example 1.3.9]{GeckJacon2011}.
\end{rem}

\begin{lem}\label{lemafunction}
For all partition $\la$, we have $n(\la) = \ba(\bla)$.
\end{lem}

\begin{proof}
The reverse combinatorial procedure to taking the twisted $2$-quotient is
explained in \cite[Section 7.2]{GerberHissJacon2014}. Accordingly, for fixed 
$t\in\N$, we denote $\Phi_t(\bla)$ the (unique) partition such that 
$\Phi_t(\bla)_{(2)}=\De_t$ and $\overline{\Phi_t(\bla)}^{(2)}=\bla$.
Therefore, we write $\la=\Phi_t(\bla)$. We remark that the construction of 
$\Phi_t(\bla)$ is done via the symbol $\fB(\bla,\bc)$, so that the parts of 
$\Phi_t(\bla)=\la$ can be read off $\fB(\bla,\bc)$: they are precisely the 
integers $2(\ka_j - \ka_j^0)$. This, together with Remark \ref{remafunction}, 
proves the claim.
\end{proof}

We are now ready to prove 
those parts of \cite[Conjecture~5.7]{GerberHissJacon2014}
which concern the Harish-Chandra series arising from cuspidal modules lifting
to cuspidal unipotent defect~$0$ modules. By formula \cite[(8.5)]{FongSrinivasan1982} 
of Fong and Srinivasan, an ordinary unipotent module labelled by the
partition~$\lambda$ has defect~$0$ if and only if~$\lambda$ is an $e$-core.
Our result generalizes \cite[Theorem~$5.4$]{GeckJacon2006} in level~$2$ as well as
\cite[Theorem~6.2]{GerberHissJacon2014}. The proof is inspired by the considerations 
in \cite[2.5]{Geck2006}. Remarkably, there is no restriction on~$\ell$.

\begin{thm}
\label{MainTheoremUnitaryGroups}
Let $0 \leq s < (e - 1)/2$ be an integer, put $m := \lfloor s(s+1)/2 \rfloor$,
and let~$X$ denote the unipotent cuspidal $kG_m$-module labelled by $\Delta_s$.
Then under the identification $X_{\la} \leftrightarrow \overline{\la}^{(2)}$,
the Harish-Chandra branching graph $\mathcal{G}_{\mathcal{D},q,\ell}( X )$
is exactly the crystal graph $\mathcal{B}_{\mathbf{s},e}$
in Uglov's realization, with $\bs = ( s + (1-e)/2, 0 )$.
\end{thm}
\begin{proof}
Fix a non-negative integer~$n$, put $L := L_{m,n}$ and $G := G_{m+n}$. Let
$r := 2(m+n) + \iota$ with $\iota \in \{0,1\}$ such that $2m + \iota = s(s+1)/2$
and $G = G_{m+n} = \GU_r(q)$. 
Choose an $\ell$-modular system $(K, \mathcal{O}, k)$ such that~$K$ is large 
enough for~$G$. By our assumption on~$s$, the triangular partition $\Delta_s$
is an $e$-core. Hence the cuspidal unipotent $KL$-module~$Y$ labelled 
by~$\Delta_s$ is of $\ell$-defect~$0$,
and thus reduces irreducibly to the $kL$-module~$X$. In particular,~$X$ is 
cuspidal.

Let $\hat{X}$ denote an $\mathcal{O}L$-lattice in~$Y$. Then~$X$ and~$\hat{X}$ 
are projective, as~$Y$ is of defect~$0$. It follows that $R_L^G( \hat{X} )$ is 
projective. Write 
$$R_L^G( X ) = X_1 \oplus X_2 \oplus \cdots \oplus X_h$$
with projective indecomposable $kG$-modules~$X_i$, $i = 1, \dots , h$. Then
$$R_L^G( \hat{X} ) = \hat{X}_1 \oplus \hat{X}_2 \oplus \cdots \oplus \hat{X}_h$$
with indecomposable projective $\mathcal{O}G$-lattices~$\hat{X}_i$ 
lifting~$X_i$, $i = 1, \ldots , h$. For each $i = 1 , \ldots , h$, the 
irreducible constituents of $K \otimes_\mathcal{O} \hat{X}_i$ are unipotent and
thus of the form $Y_\nu$ for partitions~$\nu$ of~$r$ with $2$-core~$\Delta_s$, 
as they lie in the ordinary Harish-Chandra series determined by~$(L,Y)$.

Now let $\leq$ denote the lexicographic order on the set of partitions of~$r$.
For $i = 1, \ldots , h$ let 
$$\mu(i) := 
\max\{ \nu \mid [ K \otimes_{\mathcal{O}} \hat{X}_i \colon\!Y_\nu ] \neq 0 \},$$
and put 
$$\Lambda^0 := \{ \mu(i) \mid 1 \leq i \leq h \}.$$
Let $1 \leq i, j \leq h$. Geck's result \cite{Geck1991} on the triangular shape 
of the $\ell$-modular decomposition matrix of~$\GU_r(q)$ implies that $\mu(i) = 
\mu(j)$ if and only if $X_i \cong X_j$. For each $\mu \in \Lambda^0$ choose~$i$ 
with $1 \leq i \leq h$ and $\mu = \mu(i)$ and write~$X_\mu$ for the simple head
composition factor of~$X_i$. Then the set $\{ X_\mu \mid \mu \in \Lambda^0 \}$ 
equals the Harish-Chandra series determined by $(L,X)$. 
As the elements of~$\Lambda^0$ are partitions of~$r$ with $2$-core $\Delta_s$, 
the set 
$$\bar{\Lambda}^0 := \{ \bar{\mu}^{(2)} \mid \mu \in \Lambda^0 \}$$ 
consists of bipartitions of~$n$.

For $R \in \{ K, \mathcal{O}, k \}$ put
$$H_R(L, R \otimes_{\mathcal{O}} \hat{X} ) :=
\End_{RG}( R_L^G( R \otimes_{\mathcal{O}} \hat{X} ) ).$$
Then $H_R(L, R \otimes_{\mathcal{O}} \hat{X} )$ is an Iwahori-Hecke
algebra over~$R$ of type~$B_n$ with parameters~$q^2$ and $q^{2s+1}$, viewed as
elements of~$R$ (see \cite[p.~$464$]{C}). We denote the Hom-functor with respect 
to~$R_L^G( R \otimes_{\mathcal{O}} \hat{X} )$ by~$F_R$.

For a bipartition~${\bnu}$ of~$n$ and $R \in \{K, k \}$, let $S^{\bnu}_R$
denote the Specht module of $H_R( L, R \otimes_{\mathcal{O}} \hat{X} )$
associated to~${\bnu}$ by Dipper James and Murphy \cite[Theorem~$4.22$]{DiJaMu} 
(where we follow \cite{GeckJacon2011} in our notational convention).
By the results of Fong and Srinivasan in the appendix of 
\cite{FongSrinivasan1990}, we have $F_K( Y_\nu ) = S_K^{{\bnu}}$
with ${\bnu} = \bar{\nu}^{(2)}$ for all $\nu \vdash r$ with $2$-core~$\Delta_s$.
For $\mu \in \Lambda^0$ and ${\bmu} := \bar{\mu}^{(2)}$ put $M^{\bmu} := 
F_k( X_\mu )$. Then $\{ M^{\bmu} \mid {\bmu} \in \bar{\Lambda}^0 \}$ is a set
of representatives for the simple $H_k( L, X )$-modules.

We claim that $\bar{\Lambda}^0$ is a canonical basic set for $H_k( L, X )$
as defined in~\cite[Definition~$3.2.1$]{GeckJacon2011} or 
\cite[Definition~$2.4$]{GeckJacon2006}. 
For a partition~$\nu$ of~$r$ with
$2$-core~$\Delta_s$ and
$\mu  = \mu(i) \in \Lambda^0$,  where $1 \leq i \leq h$, put 
$$d_{\nu,\mu} := [ K \otimes_{\mathcal{O}} \hat{X}_i\colon\!Y_\nu ].$$
By the result of Geck \cite{Geck1991}, $d_{\nu,\mu} \neq 0$ implies that 
either $\nu = \mu$ and $d_{\mu,\mu} = 1$, or else that~$\nu$ is strictly 
smaller than~$\mu$
in the dominance order on partitions. The latter implies that
$n( \mu ) < n( \nu )$ (see e.g.\ \cite[Exercise~$5.6$]{GeckPfeiffer2000}). 
Thus $d_{\nu,\mu} \neq 0$ implies that $\nu = \mu$ or $n(\mu) < n(\nu)$.
Now
$$d_{\nu,\mu} = [ K \otimes_{\mathcal{O}} \hat{X}_i\colon\!Y_\nu ]
= [ F_K( K \otimes_{\mathcal{O}} \hat{X}_i )\colon\!F_K( Y_\nu ) ]
= [ S_k^{\bnu}\colon\!M^{\bmu} ]$$
by Brauer reciprocity applied to $(H_K( L, Y ), H_{\mathcal{O}}( L, \hat{X} ),
H_k( L, X ) )$, as $F_K( Y_\nu ) = S_K^{\bnu}$ and 
$F_{\mathcal{O}}( \hat{X}_i )$ is the lift of the projective cover of
$F_k( X_{\mu} ) = M^{\bmu}$. 
Now $n( \nu ) = \mathbf{a}({\bnu})$ by Lemma \ref{lemafunction}.
Thus $[ S_k^{\bmu}\colon\!M^{\bmu} ] = 1$
and $[ S_k^{\bnu}\colon\!M^{\bmu} ] \neq 0$ implies that ${\bnu} = {\bmu}$
or $\mathbf{a}({\bmu}) < \mathbf{a}({\bnu})$. Hence $\bar{\Lambda}^0$ is a
canonical basic set for $H_k(L,X)$ as claimed.

It now follows from
\cite[Lemma~$5.2$ and Example~$5.6$]{GeckJacon2006} that $\bar{\Lambda}^0$ 
equals the set $\cU_{\bs,e}$ with $\bs = (s + (1-e)/2,0)$, thus proving our 
assertion.
\end{proof}

\section*{Acknowledgements}

It is our pleasure to thank Olivier Dudas, Meinolf Geck, Nicolas Jacon
and Peng Shan for many useful discussions on the results of this paper.



\end{document}